\documentclass[10pt]{article}
\usepackage{amsfonts,a4,latexsym}
\usepackage[utf8]{inputenc}
\usepackage{amsmath} 
\usepackage{amsthm,amscd,amssymb} 
\usepackage[english]{babel} 
\usepackage{bm}
\usepackage[all]{xy}
\usepackage{tikz}
\usepackage{paralist}
\usepackage{bbm}
\usepackage{dsfont} 
\usepackage[OT2,OT1]{fontenc}
\newcommand{\cyr}{%
\renewcommand\rmdefault{wncyr}%
\renewcommand\sfdefault{wncyss}%
\renewcommand\encodingdefault{OT2}%
\normalfont\selectfont}
\DeclareTextFontCommand{\textcyr}{\cyr}
\usepackage[left=3cm,right=3cm,top=3cm,bottom=3cm]{geometry} 
\newtheoremstyle{break}
  {9pt}
  {9pt}
  {\itshape}
  {}
  {\bfseries}
  {}
  {\newline}
  {}
\newtheoremstyle{beispiel}
  {9pt}
  {9pt}
  {\upshape}
  {}
  {\bfseries}
  {}
  {\newline}
  {}
\theoremstyle{break}
\newtheorem{defi}{Definition}[section]
\newtheorem{prop}[defi]{Proposition}
\newtheorem{lemma}[defi]{Lemma}
\newtheorem{cor}[defi]{Corollary}
\newtheorem{theorem}[defi]{Theorem}
\theoremstyle{beispiel}
\newtheorem{example}[defi]{Example}

\DeclareMathOperator{\bk}{\mathbf{k}}

\DeclareMathOperator{\rmod}{\mathrm{mod-}}

\DeclareMathOperator{\Ext}{\mathrm{Ext}}

\DeclareMathOperator{\rad}{\mathrm{rad}}

\DeclareMathOperator{\pd}{\mathrm{pdim}}

\DeclareMathOperator{\N}{\mathbb{N}}
\DeclareMathOperator{\sQ}{\mathcal{Q}}
\DeclareMathOperator{\Ld}{\Lambda}

\DeclareMathOperator{\ld}{\lambda}

\begin{document}
\title{The strong no loop conjecture is true for mild algebras}

\author{Denis Skorodumov\\ Bergische Universit\"at Wuppertal\\ Germany\thanks{E-mail: skorodumov@math.uni-wuppertal.de}}
\date{November 04, 2010}

\maketitle
\begin{center}\textcyr{Dlya moikh roditele\u i Reginy i Vitaliya}\end{center}\vspace{10pt}

\begin{abstract}
Let $\Lambda$ be a finite dimensional associative algebra over an algebraically closed field with a simple module $S$ of finite projective dimension. The strong no loop conjecture says that this implies $\mathrm{Ext}^1_{\Lambda}(S,S)=0$, i.e. that the quiver of $\Lambda$ has no loops in the point corresponding to $S$. In this paper we prove the conjecture in case $\Lambda$ is mild, which means that $\Lambda$ has only finitely many two-sided ideals and each proper factor algebra $\Lambda/J$ is representation finite. In fact, it is sufficient that a ''small neighborhood'' of the support of the projective cover of $S$ is mild.
\end{abstract}

\section{Introduction}
Let $\Ld$ be a finite dimensional associative algebra over a fixed algebraically closed field $\bk$ of arbitrary characteristic. We consider only $\Ld$-right modules of finite dimension.

The strong no loop conjecture says that a simple $\Ld$-module $S$ of finite projective dimension satisfies $\Ext^1_{\Ld}(S,S)=0$. To prove this conjecture for a given algebra we can switch to the Morita-equivalent basic algebra and therefore assume that $\Ld=\bk\sQ/I$ for some quiver $\sQ$ and some ideal $I$ generated by linear combinations of paths of length at least two. Then $S=S_x$ is the simple corresponding to a point $x$ in $\sQ$ and the conjecture means that there is no loop at $x$ provided the projective dimension $\pd_{\Ld}S_x$ is finite.\\

The conjecture is known for
\begin{itemize}
        \item monomial algebras by Igusa \cite{Igu90},
        \item truncated extensions of semi-simple rings by Marmaridis, Papistas \cite{MP95},
        \item bound quiver algebras $k\sQ/I$ such that for each loop $\alpha\in\sQ$ there exists an $n\in \N$ with $\alpha^n\in I\setminus (IJ+JI)$, where $J$ denotes the ideal generated by the arrows \cite{GSZ01},
        \item special biserial algebras by Liu, Morin \cite{LM04},
        \item two point algebras with radical cube zero by Jensen \cite{Jen05}.
\end{itemize}

In this paper, we prove the conjecture for another class of algebras including all representation-finite algebras. To state our result precisely we introduce for any point $x$ in $\sQ$ its \textbf{neighborhood} $\Ld(x)=e\Ld e$. Here $e$ is the sum of all primitive idempotents $e_z\in\Ld$ such that $z$ belongs to the support of the projective $P_x:=e_x\Ld$ or such that there is an arrow $z\to x$ in $\sQ$ or a configuration $y'\leftarrow x\rightleftarrows y\leftarrow z$ with 4 different points $x,y,y'$ and $z$.\\
Recall that an algebra $\Ld$ is called \textbf{distributive} if it has a distributive lattice of two-sided ideals and \textbf{mild} if it is distributive and any proper quotient $\Ld/J$ is representation-finite.

Our main result reads as follows:
\begin{theorem}\label{maintheorem} Let $\Ld=\bk\sQ/I$ be a finite dimensional algebra over an algebraically closed field $\bk$. Let  $x$ be a point in $\sQ$ such that the corresponding simple $\Ld$-module $S_x$ has finite projective dimension. If $\Ld(x)$ is mild, then there is no loop at $x$.
\end{theorem}
Of course, it follows immediately that the strong no loop conjecture holds for all mild algebras, in particular for all representation-finite algebras.
\begin{cor}\label{maincor}
Let $\Ld$ be a mild algebra over an algebraically closed field. Let $S$ be a simple $\Ld$-module. If the projective dimension of $S$ is finite, then $\Ext^1_{\Ld}(S,S)=0$.
\end{cor}
In order to prove the theorem we do not look at projective resolutions. Instead we refine a little bit the K-theoretic arguments of Lenzing \cite[Satz 5]{Len69}, also used by Igusa in his proof of the strong no loop conjecture for monomial algebras \cite[Corollary 6.2]{Igu90}, to obtain the following result:
\begin{prop}\label{lenzingsresult}
    Let $\Ld=\bk\sQ/I$ be a finite dimensional algebra, $x$ a point in $\sQ$ and $\alpha$ an oriented cycle at $x$. If $P_x$ has an $\alpha$-filtration of finite projective dimension, then $\alpha$ is not a loop.
\end{prop}
Here an $\alpha$\textbf{-filtration} $\mathcal{F}$ of $P_x$ is a filtration
\[P_x=M_0\supset M_1\supset\ldots\supset M_n=0\]
by submodules with
\[\alpha M_i\subset M_{i+1}\ \forall\ i=0\ldots n-1 .\]
The filtration $\mathcal{F}$ has finite projective dimension if $\pd_{\Ld} M_i<\infty$ holds for all $i=1\ldots n-1$.\par
This proposition is shown by Lenzing in \cite[Satz 5]{Len69} for the special filtration $M_i=\alpha^{i}\Ld$, but his proof remains valid for all $\alpha$-filtrations.

Our strategy to prove Theorem \ref{maintheorem} is then as follows: We consider the point $x$ with $\pd_{\Ld}S_x<\infty$ and its mild neighborhood $A:=\Ld(x)$. We assume in addition that there is a loop $\alpha$ in $x$. Then we deduce a contradiction either by showing that $\pd_{\Ld} S_x=\infty$ or by constructing a certain $\alpha$-filtration $\mathcal{F}$ of $P_x$ having finite projective dimension in $\rmod\Ld$ and implying that $\alpha$ is not a loop by Proposition \ref{lenzingsresult}. Since $\Ld(x)$ contains the support of $P_x$, this filtrations coincide for $P_x$ as a $\Ld$-module and as a $\Ld(x)$-module. Thus we are dealing with a mild algebra, and we use in an essential way the deep structure theorems about such algebras given in \cite{BGRS85} and \cite{B09} to obtain the wanted $\alpha$-filtrations. In particular, we show that we always work in the ray-category attached to $\Ld(x)$. This makes it much easier to use cleaving diagrams. But still the construction of the appropriate $\alpha$-filtrations depends on the study of several cases and it remains a difficult technical problem. The $\alpha$-filtrations are always built in such a way that they have finite projective dimension in $\rmod\Ld$ provided $\pd_{\Ld}S_x<\infty$.\par

To illustrate the method by two examples we define $\langle w_1,\ldots, w_k\rangle$ as the submodule of $P_x$ generated by elements $w_1,\ldots ,w_k\in P_x$.
\begin{example}
Let $\Ld$ be an algebra such that $\Ld(x)$ is given by the quiver
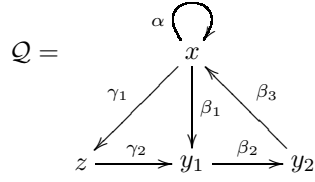
\[\sQ=
	\xymatrix@!=1.5pc{
    & x\ar@(lu,ru)[]^(.2){\alpha}\ar[d]^{\beta_1}\ar[dl]_{\gamma_1} & \\
	z\ar[r]^{\gamma_2} & y_1\ar[r]^{\beta_2} & y_2 \ar[ul]_{\beta_3}
	}\ \]
and a relation ideal $I$ such that the projective module $P_x$ is described by the following graph:
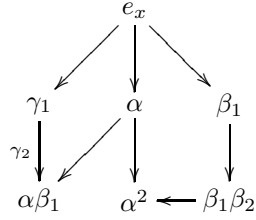
\[\xymatrix@!=1pc{
& e_x \ar[dr] \ar[dl] \ar[d] & \\
\gamma_1\ar[d]_{\gamma_2} & \alpha\ar[d]\ar[dl] & \beta_1\ar[d] \\
\alpha\beta_1  & \alpha^2 & \beta_1\beta_2\ar[l]
}.\]
Notice that the picture means that there are relations $\alpha^2-\ld_1\beta_1\beta_2\beta_3,\ \alpha\beta_1-\ld_2\gamma_1\gamma_2 \in I$ for some $\ld_i\in\bk\setminus\{0\}$. From the obvious exact sequences
\[0\to\rad P_x\to P_x\to S_x\to 0\]
\[0\to \langle\beta_1,\gamma_1\rangle\to\rad P_x\to S_x\to 0\]
\[0\to \langle\alpha^2,\gamma_1\rangle\to\langle\alpha,\gamma_1\rangle\to S_x\to 0\]
we see that $\pd_{\Ld}S_x<\infty$ leads to $\pd_{\Ld}\rad P_x<\infty$ and $\pd_{\Ld}\langle\beta_1,\gamma_1\rangle<\infty$. Since $\langle\beta_1,\gamma_1\rangle=\langle\beta_1\rangle\oplus\langle\gamma_1\rangle$ and $\langle\alpha^2,\gamma_1\rangle=\langle\alpha^2\rangle\oplus\langle\gamma_1\rangle$ in this example, both $\pd_{\Ld}\langle\gamma_1\rangle$ and $\pd_{\Ld}\langle\alpha,\gamma_1\rangle$ are finite.
Then the following $\alpha$-filtration $\mathcal{F}$: $P_x\supset \langle \alpha,\gamma_1 \rangle \supset \langle \alpha^2 \rangle \supset 0$ has finite projective dimension in $\rmod\Ld$.
\end{example}

In the next example we see that this method may not work if the neighborhood $\Ld(x)$ is not mild, even if the support of $P_x$ is mild.
\begin{example}
  Let $\Ld(x)=\bk\sQ/I$ be given by the quiver \[\sQ=
  
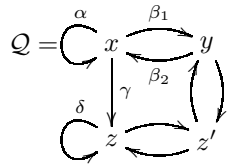
\xymatrix@!=1pc{
  x\ar@(ul,dl)[]_(.2){\alpha}\ar@/^0.5pc/[r]^{\beta_1} \ar[d]^\gamma & y\ar@/^0.5pc/[l]^{\beta_2}\ar@/^0.5pc/[d] \\
  z\ar@(ul,dl)[]_(.2){\delta}\ar@/^0.5pc/[r] & z'\ar@/^0.5pc/[l]\ar@/^0.5pc/[u]
  }\] and by a relation ideal $I$ such that $P_x$ is represented by
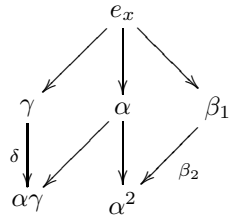
\[\xymatrix@!=1pc{
& e_x \ar[dr] \ar[dl] \ar[d] & \\
\gamma\ar[d]_\delta & \alpha \ar[d]\ar[dl] & \beta_1 \ar[dl]^{\beta_2} \\
\alpha\gamma  & \alpha^2 &
}.\]
Here we get stuck because the uniserial module with basis $\{\gamma,\alpha\gamma\}$ allows only the composition series as an $\alpha$-filtration. Since we do not know $\pd_{\Ld}S_z$, which depends on $\Ld$ and not only on $\Ld(x)$, our method does not apply.
\end{example}

The article is organized as follows: In the second section we recall some facts about ray-categories and we show how to reduce the proof to standard algebras without penny-farthings. This case is then analyzed in the last section.

The results of this article are contained in my PhD-thesis written at the University of Wuppertal.

\textbf{Acknowledgment:} I would like to thank Klaus Bongartz for his support and for very helpful discussions.

\section{The reduction to standard algebras}
\subsection{Ray-categories and standard algebras}
We recall some well-known facts from \cite{BGRS85}, \cite{GR92}.

Let $A:=\Ld(x)=\bk\sQ_A/I_A$ be a basic distributive $\bk$-algebra. Then every space $e_xAe_y$ is a cyclic module over $e_xAe_x$ or $e_yAe_y$ and we can associate to $A$ its \textbf{ray-category} $\overrightarrow{A}$. Its objects are the points of $\sQ_A$. The morphisms in $\overrightarrow{A}$ are called \textbf{rays} and $\overrightarrow{A}(x,y)$ consists of the orbits $\overrightarrow{\mu}$ in $e_xAe_y$ under the obvious action of the groups of units in $e_xAe_x$ and $e_yAe_y$. The composition of two morphisms $\overrightarrow{\mu}$ and $\overrightarrow{\nu}$ is either the orbit of the composition $\mu\nu$, in case this is independent of the choice of representatives in $\overrightarrow{\mu}$ and $\overrightarrow{\nu}$, or else $0$. We call a non-zero morphism $\eta\in \overrightarrow{A}$ \textbf{long} if it is non-irreducible and satisfies $\nu\eta=0=\eta\nu'$ for all non-isomorphisms $\nu,\nu'\in \overrightarrow{A}$. One crucial fact about ray-categories frequently used in this paper is that $A$ is mild iff $\overrightarrow{A}$ is so \cite[see Theorem 13.17]{GR92}.

The ray-category is a finite category characterized by some nice properties. For instance, given $\ld\mu\kappa=\ld\nu\kappa\neq 0$ in $\overrightarrow{A}$, $\mu=\nu$ holds. We shall refer to this property as the \textbf{cancellation law}.

Given $\overrightarrow{A}$, we construct in a natural way its linearization $\bk(\overrightarrow{A})$ and obtain a finite dimensional algebra \[\overline{A}=\bigoplus_{x,y\in\sQ_A}\bk(\overrightarrow{A})(x,y),\] the \textbf{standard form} of $A$. In general, $A$ and $\overline{A}$ are not isomorphic, but they are if either $A$ is minimal representation-infinite \cite[Theorem 2]{B09} or representation-finite with $\mathrm{char} \bk\neq 2$ \cite[Theorem 13.17]{GR92}.\\
Similar to $A$, the ray-category $\overrightarrow{A}$ admits a description by quiver and relations. Namely, there is a canonical full functor ${}^{\to}: \mathcal{P}\sQ_A \to \overrightarrow{A}$ from the path
category of $\sQ_A$ to $\overrightarrow{A}$.
Two paths in $\sQ_A$ are \textbf{interlaced} if they belong to the transitive closure of the
relation given by $v\sim w$ iff $v = pv'q,\ w = pw'q$ and $\overrightarrow{v'} = \overrightarrow{w'} \neq 0$, where $p$
and $q$ are not both identities.\\
A \textbf{contour} of $\overrightarrow{A}$ is a pair $(v, w)$ of non-interlaced paths with $\overrightarrow{v} = \overrightarrow{w} \neq 0$. Note that these contours are called essential contours in \cite[2.7]{BGRS85}. Throughout this paper we will need a special kind of contours called penny farthings. A \textbf{penny-farthing} $P$ in $\overrightarrow{A}$ is a contour $(\sigma^2,\rho_1\ldots\rho_s)$ such that the full subquiver $\sQ_P$ of $\sQ_A$ that supports the arrows of $P$ has the following shape:

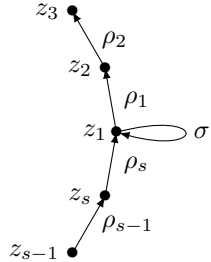
\[\begin{tikzpicture}[scale=2.5,cap=round,->,>=latex]
        \path (0:1.0cm) edge [loop right] node {$\sigma$} (0:1.0cm);
        \draw[->] (20:1cm) to node[right=0.05cm] {$\rho_2$} (40:1cm) node {$\bullet$} node[left=0.05cm] {$z_3$};
        \draw[->] (0:1cm) node {$\bullet$} node[left=0.05] {$z_1$} to node[right=0.05cm] {$\rho_1$} (20:1cm) node {$\bullet$} node[left=0.05cm] {$z_2$};
        \draw[->] (340:1cm) node {$\bullet$} node[left=0.05cm] {$z_s$} to node[right=0.05cm] {$\rho_s$}(360:1cm);
        \draw[->] (320:1cm) node {$\bullet$} node[left=0.05cm] {$z_{s-1}$} to node[right=0.05cm] {$\rho_{s-1}$}(340:1cm);
        \foreach \x in {40,60,...,320} {
                \filldraw[black] (\x:1cm) circle(0.1pt);
        }
\end{tikzpicture}\]
Moreover, we ask the full subcategory $A_P\subset A$ living on $\sQ_P$ to be defined by $\sQ_P$ and one of the following two systems of relations
\begin{eqnarray}
  0=\sigma^2-\rho_1\ldots\rho_s &=& \rho_s\rho_1=\rho_{i+1}\ldots\rho_s\sigma\rho_1\ldots\rho_{f(i)},\\
  0=\sigma^2-\rho_1\ldots\rho_s &=& \rho_s\rho_1-\rho_s\sigma\rho_1=\rho_{i+1}\ldots\rho_s\sigma\rho_1\ldots\rho_{f(i)},
\end{eqnarray}
where $f:\{1,2,\ldots , s-1\}\to\{1,2,\ldots , s\}$ is some non-decreasing function (see \cite[2.7]{BGRS85}. For penny-farthings of type (1) $A_P$ is standard, for that of type (2) $A_P$ is not standard in case the characteristic is two.
\par
A functor $F : D \to \overrightarrow{A}$ between ray categories is \textbf{cleaving} (\cite[13.8]{GR92}) iff
it satisfies the following two conditions and their duals:
\begin{itemize}
\item[a)] $F(\mu) = 0$ iff $\mu = 0$.
\item[b)] If $\eta\in D(y, z)$ is irreducible and $F (\mu) : F (y) \to F (z')$ factors through $F (\eta)$ then $\mu$
factors already through $\eta$.
\end{itemize}
The key fact about cleaving functors is that $\overrightarrow{A}$ is not
representation finite if $D$ is not.
In this article $D$ will always be given by its quiver $\sQ_D$, that has no oriented
cycles and some relations. Two paths between the same points give always
the same morphism, and zero relations are indicated by a dotted line. As in \cite[section 13]{GR92}, the cleaving functor is then defined by drawing the quiver of $D$ with
relations and by writing the morphism $F(\mu)$ in $\overrightarrow{A}$ close to each arrow $\mu$.
\par
By abuse of notation, we denote the irreducible rays of $\overrightarrow{A}$ and the corresponding arrows of $\sQ_A$ by the same letter.

\subsection{Getting rid of penny-farthings}\label{section_penny}
Using the above notations let $P=(\sigma^2,\rho_1\ldots\rho_s)$ be a penny-farthing in $\overrightarrow{A}$. We shall show now that $x=z_1$. Therefore $\sigma=\alpha$ and $P$ is the only penny-farthing in $\overrightarrow{A}$ by \cite[Theorem 13.12]{GR92}.

\begin{lemma}\label{L2.1}
If there is a penny-farthing $P=(\sigma^2,\rho_1\ldots\rho_s)$ in $\overrightarrow{A}$, then $z_1=x$.
\end{lemma}
\begin{proof} We consider two cases:
\begin{itemize}
\item $x\in\sQ_{P}$: Hence $\sQ_{P}$ has the following shape:
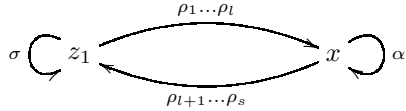
\[\xymatrix@!=2pc{
z_1 \ar@(ul,dl)[]_{\sigma} \ar@/^1pc/[rr]^{\rho_1\ldots \rho_l} & & x \ar@/^1pc/[ll]^{\rho_{l+1}\ldots \rho_s} \ar@(ur,dr)[]^{\alpha}
}\]
But this can be the quiver of a penny-farthing only for $z_1=x$.
\item $x\notin\sQ_{P}$: Since $A$ is the neighborhood of $x$, only the following cases are possible:
\begin{itemize}
\item[a)] $e_x A e_z\neq 0$: Since $x\notin \sQ_{P}$ we can apply the dual of \cite[Theorem 1]{B85} or \cite[Lemma 13.15]{GR92} to $\overrightarrow{A}$ and we see that the following quivers occur as subquivers of $\sQ_A$:
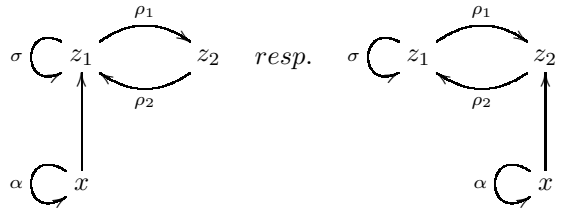
\[\xymatrix@!=2pc{
z_1 \ar@(ul,dl)[]_{\sigma} \ar@/^1pc/[r]^{\rho_1} & z_2 \ar@/^1pc/[l]^{\rho_2}\\
x \ar@(ul,dl)[]_{\alpha} \ar[u] &
}\ \ \ resp.\ \ \
\xymatrix@!=2pc{
z_1 \ar@(ul,dl)[]_{\sigma} \ar@/^1pc/[r]^{\rho_1} & z_2 \ar@/^1pc/[l]^{\rho_2}\\
&  x \ar@(ul,dl)[]_{\alpha} \ar[u]
}.\]
Moreover, there can be only one arrow starting in $x$. This is a contradiction to the actual setting.

\item[b)] $\exists\ z_1 \to x$: By applying \cite[Theorem 1]{B85} or the dual of \cite[Lemma 13.15]{GR92} we deduce that the following quiver occurs as a subquiver of $\sQ_A$:
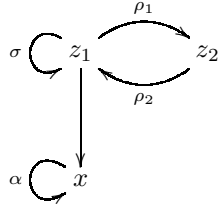
\[\xymatrix@!=2pc{
z_1 \ar@(ul,dl)[]_{\sigma} \ar@/^1pc/[r]^{\rho_1} \ar[d] & z_2 \ar@/^1pc/[l]^{\rho_2}\\
x \ar@(ul,dl)[]_{\alpha} &
}\]
and there can be only one arrow ending in $x$ contradicting the present case.

\item[c)] $\exists\ y'\leftarrow x\rightleftarrows y\leftarrow z_1$: If $y\notin \sQ_P$, then
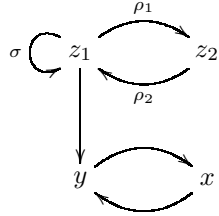
\[\xymatrix@!=2pc{
z_1 \ar@(ul,dl)[]_{\sigma} \ar@/^1pc/[r]^{\rho_1} \ar[d] & z_2 \ar@/^1pc/[l]^{\rho_2}\\
y \ar@/^1pc/[r] & x \ar@/^1pc/[l]
}\] is a subquiver of $\sQ_A$ leading to the same contradiction as in b).\\
If $y\in\sQ_P$, then $y=z_2$ and the quiver
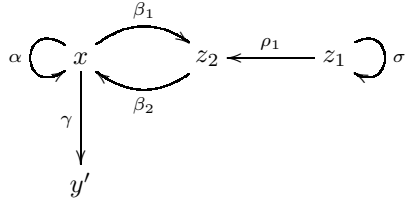
\[\xymatrix@!=2pc{
x \ar@(ul,dl)[]_{\alpha} \ar@/^1pc/[r]^{\beta_1} \ar[d]_{\gamma} & z_2\ar@/^1pc/[l]^{\beta_2} & z_1\ar@(ur,dr)[]^{\sigma} \ar[l]_{\rho_1}\\
y' & &
}\] is a subquiver of $\sQ_A$. Since $x\notin\sQ_P$, all morphisms occurring in the following diagram
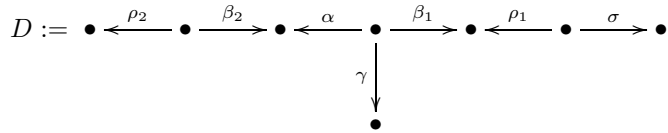
\[D:=\xymatrix@!=1pc{
\bullet & \bullet\ar[l]_{\rho_2}\ar[r]^{\beta_2} & \bullet & \bullet\ar[l]_\alpha\ar[d]_{\gamma}\ar[r]^{\beta_1} & \bullet & \bullet \ar[l]_{\rho_1}\ar[r]^{\sigma} & \bullet \\
 &  &  & \bullet  & & & }\]
are irreducible and pairwise distinct. Therefore $D$ is a cleaving diagram in $\overrightarrow{A}$. Moreover, some long morphism $\eta=\nu\sigma^3\nu'$ does not occur in $D$; hence $D$ is still cleaving in $\overrightarrow{A}/\eta$ by \cite[Lemma 3]{B09}. Since $D$ is of representation-infinite Euclidean type $\widetilde{E}_7$, $\overrightarrow{A}/\eta$ is representation-infinite contradicting the mildness of $A$.
\end{itemize}
\end{itemize}
\end{proof}

Now, we show that, provided the existence of a penny-farthing in $\overrightarrow{A}$, there exists an $\alpha$-filtration of $P_x$ having finite projective dimension.

\begin{lemma}\label{L2.2}
Let $A=\Ld(x)$ be mild and standard. If there is a penny-farthing in $\overrightarrow{A}$, then there exists an $\alpha$-filtration $\mathcal{F}$ of $P_x$ having finite projective dimension.
\end{lemma}
\begin{proof} If there is a penny-farthing $P$ in $\overrightarrow{A}$, then $P=(\alpha^2,\rho_1\ldots \rho_s)$ is the only penny-farthing in $\overrightarrow{A}$ by the last lemma. Since $A$ is standard and mild, there are three cases for the graph of $P_x$ which can occur by \cite[Theorem 1]{B85} or the dual of \cite[Lemma 13.15]{GR92}.
\begin{itemize}
\item[I)] There exists an arrow $\gamma:x\to z,\ \gamma\neq \rho_1$. Then $s=2$, the quiver
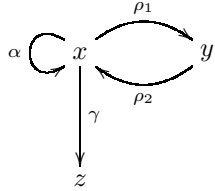
\[\xymatrix@!=2pc{
x \ar@(ul,dl)[]_{\alpha} \ar@/^1pc/[r]^{\rho_1}  \ar[d]^{\gamma} & y \ar@/^1pc/[l]^{\rho_2} \\
z &
}\] is a subquiver of $\sQ_A$, and $P_x$ is represented by the following graph:
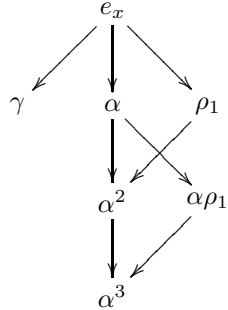
\[\xymatrix@!=1pc{
  & e_x \ar[dr] \ar[d] \ar[dl] &\\
\gamma & \alpha \ar[d] \ar[dr] & \rho_1 \ar[dl]\\
  & \alpha^2 \ar[d] & \alpha\rho_1 \ar[dl]\\
  & \alpha^3 &
}.\]
Let $M$ be a quotient of $P_x$ defined by the following exact sequence:
\[0 \to \langle\gamma\rangle \oplus \langle\rho_1,\alpha\rho_1\rangle \to P_x \to M \to 0.\]
Then $M$ has $S_x$ as the only composition factor. Hence $\pd_{\Ld}M<\infty$ and $\pd_{\Ld}\langle \rho_1,\alpha\rho_1\rangle<\infty$. Now, we consider the exact sequence
\[0\to \langle\alpha^3\rangle \to \langle\rho_1,\alpha\rho_1\rangle \to \langle\rho_1\rangle/\langle \alpha^3\rangle \oplus \langle\alpha\rho_1\rangle/\langle\alpha^3\rangle \to 0.\]
But $\langle\alpha^3\rangle\cong S_x$ and $\pd_{\Ld}S_x<\infty$, hence  $\langle\alpha\rho_1\rangle/\langle\alpha^3\rangle\cong S_y$ has finite projective dimension in $\rmod\Ld$. Finally, the $\alpha$-filtration $P_x \supset \langle\alpha\rangle \supset \langle\alpha^2\rangle \supset \langle\alpha^3\rangle\supset 0$ has finite projective dimension since all filtration modules $\neq P_x$ have $S_x$ and $S_y$ as the only composition factors.

\item[II)] In the second case there exists a point $z\notin\sQ_P$ such that $A(x,z)\neq 0$. Then $s=2$, the quiver
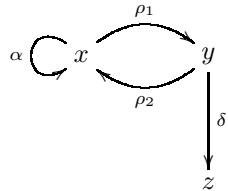
\[\xymatrix@!=2pc{
x \ar@(ul,dl)[]_{\alpha} \ar@/^1pc/[r]^{\rho_1} & y \ar@/^1pc/[l]^{\rho_2} \ar[d]^{\delta} \\
 & z
}\] is a subquiver of $\sQ_A$, and $P_x$ is represented by:
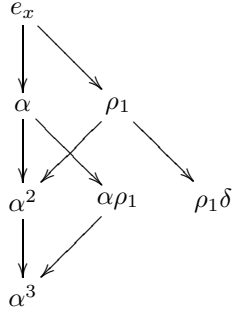
\[\xymatrix@!=1pc{
e_x \ar[dr] \ar[d] & &\\
\alpha \ar[d] \ar[dr] & \rho_1 \ar[dl] \ar[dr] & \\
\alpha^2 \ar[d] & \alpha\rho_1 \ar[dl] & \rho_1\delta\\
\alpha^3 &
}.\]
With similar considerations as in I) we obtain that the same filtration fits.

\item[III)] In the last possible case we have $A(x,z)=0$ for all points $z\notin\sQ_P$. Hence $P_x$ is represented by:
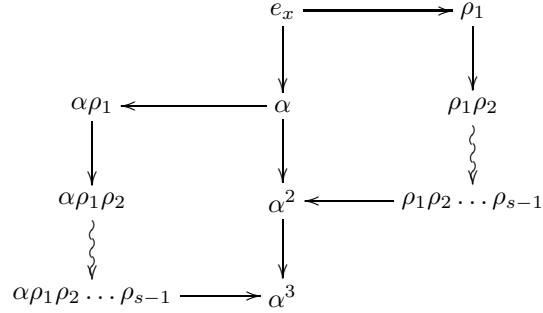
\[\xymatrix@!=1pc{
& & e_x \ar[d] \ar[rr] & & \rho_1 \ar[d]\\
\alpha\rho_1 \ar[d]& & \alpha \ar[d] \ar[ll] &  & \rho_1\rho_2 \ar@{~>}[d] \\
\alpha\rho_1\rho_2 \ar@{~>}[d]& & \alpha^2 \ar[d] &  & \rho_1\rho_2\ldots \rho_{s-1} \ar[ll]\\
\alpha\rho_1\rho_2\ldots \rho_{s-1} \ar[rr] & & \alpha^3 & &
}.\]
As a $\Ld$-Module, $M:=P_x/\langle\alpha^2\rangle$ has finite projective dimension since $\langle\alpha^2\rangle$ has $S_x$ as the only composition factor. Let $K$ be the kernel of the epimorphism
$M\to\langle\alpha^2\rangle,\ e_x\mapsto\alpha^2$, then $K=\langle\rho_1\rangle/\langle\alpha^2\rangle \oplus \langle\alpha\rho_1\rangle/\langle\alpha^3\rangle$ has finite projective dimension. Moreover, $\pd_{\Ld}\langle\rho_1\rangle, \pd_{\Ld}\langle\alpha\rho_1\rangle<\infty$.  Since
\[0\to \langle\alpha\rho_1\rangle\to \langle\alpha\rangle \overset{\ld_{\alpha}}{\to} \langle\alpha^2\rangle\to 0\] is exact, $\pd_{\Ld}\langle\alpha\rangle<\infty$. Thus the same filtration as in the first two cases fits again.
\end{itemize}
\end{proof}

\begin{lemma}\label{L2.3}
With above notations let $A=\Ld(x)$ be mild and non-standard. There exists an $\alpha$-filtration $\mathcal{F}$ of $P_x$ having finite projective dimension.
\end{lemma}
\begin{proof}
If $A$ is non-standard, then $A$ is representation finite by \cite{B09}, $\mathrm{char} \bk=2$ and there is a penny-farthing in $\overrightarrow{A}$ by \cite[Theorem 13.17]{GR92}. Since Lemma \ref{L2.1} remains valid, the penny-farthing $(\alpha^2,\rho_1\ldots \rho_s)$, $\rho_i:z_i\to z_{i+1}$, $z_1=z_{s+1}=x$, is unique. By \cite[13.14, 13.17]{GR92} the difference between $A$ and $\overline{A}$ in the composition of the arrows shows up in the graphs of the projectives to $z_2,\ldots, z_s$ only. Thus the graph of $P_x$ remains the same in all three cases of the proof of Lemma \ref{L2.2} and the filtrations constructed there still do the job.
\end{proof}

\section{The proof for standard algebras without penny-farthings}\label{section_nopenny}
\subsection{Some preliminaries}
If there is no penny-farthing in $\overrightarrow{A}$, then $A=\overline{A}$ is standard by Gabriel, Roiter \cite[Theorem 13.17]{GR92} and Bongartz \cite[Theorem 2]{B09}. By a result of Liu, Morin \cite[Corollary 1.3]{LM04}, deduced from a proposition of Green, Solberg, Zacharia \cite{GSZ01}, a power of $\alpha$ is a summand of a polynomial relation in $I=I_{\Ld}$. Otherwise $\pd_{\Ld}S_x$ would be infinite contradicting the choice of $x$. Furthermore, $\alpha$ is a summand of a polynomial relation in $I_A$ by definition of $A$. But $I_A$ is generated by paths and differences of paths in $\sQ_A$. Hence we can assume without loss of generality that there is a relation $\alpha^t-\beta_1\beta_2\ldots \beta_r$ in $I_A$ for some $t\in\N$ and arrows $\beta_1,\beta_2,\ldots ,\beta_r$. Among all relations of this type we choose one with minimal $t$. Hence $(\alpha^t, \beta_1\beta_2\ldots \beta_r)$ is a contour in $\overrightarrow{A}$ with $t,r \geq 2$. Let $y=e(\beta_1)$ be the ending point of $\beta_1$ and $\tilde{\beta}=\beta_2\ldots \beta_r$.

By the structure theorem for non-deep contours in \cite[6.4]{BGRS85} the contour $(\alpha^t, \beta_1\beta_2\ldots \beta_r)$ is deep, i.e. we have $\alpha^{t+1}=0$ in $A$. Since $A$ is mild, the cardinality of the set $x^+$ of all arrows starting in $x$ is bounded by three. Before we consider the cases $|x^+|=2$ and $|x^+|=3$ separately we shall prove some useful general facts.

The following trivial fact about standard algebras will be essential hereafter.
\begin{lemma}\label{L1.9}
Let $A=\overline{A}$ be a standard $\bk$-algebra. Consider rays $v_i, w_j\in \overrightarrow{A}\setminus \{0\}$ for $i=1\ldots n$ and $j=1\ldots m$ such that $v_l\neq v_k$ and $w_l\neq w_k$ for $l\neq k$. If there are $\lambda_i,\mu_j\in\bk\setminus\{0\}$ such that $\sum_{i=1}^n\lambda_i v_i=\sum_{j=1}^m\mu_j w_j$, then $n=m$ and there exists a permutation $\pi\in S(n)$ such that $v_i=w_{\pi(i)}$ and $\lambda_i=\mu_{\pi(i)}$ for $i=1\ldots n$.
\end{lemma}
\begin{proof}
Since the set of non-zero rays in $\overrightarrow{A}$ forms a basis of $A$, it is linearly independent and the claim follows.
\end{proof}

In what follows we denote by $\mathcal{L}$ the set of all long morphisms in $\overrightarrow{A}$. By $\mu$ we denote some long morphism $\nu\alpha^t \nu'$ which exists since $\alpha^t\neq 0$.

\begin{lemma}\label{L3.4}
Using the above notations we have: \[\langle \beta_1 \rangle \cap \langle \alpha\beta_1\rangle =0\]
\end{lemma}
\begin{proof}
We assume to the contrary that $\langle \beta_1 \rangle \cap \langle \alpha\beta_1 \rangle \neq 0$. Then, by Lemma \ref{L1.9}, there are rays $v,w\in \overrightarrow{A}$  such that $\beta_1 v=\alpha\beta_1 w\neq 0$. We claim that
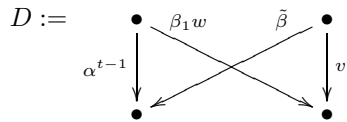
\[D:=\xymatrix@!=1pc{
\bullet \ar[d]_{\alpha^{t-1}} \ar[drr]^(.2){\beta_1 w} & & \bullet \ar[d]^v \ar[dll]_(.2){\tilde{\beta}} \\
\bullet & & \bullet }\] is a cleaving diagram in $\overrightarrow{A}$. It is of representation-infinite, Euclidean type $\widetilde{A}_3$. Since all morphisms occurring in $D$ are not long, the long morphism $\mu=\nu\alpha^t\nu'$ does not occur in $D$ and $D$ is still cleaving in $\overrightarrow{A}/\mu$ by \cite[Lemma 3]{B09}. Thus $\overrightarrow{A}/\mu$ is representation-infinite contradicting the mildness of $A$.

Now we show in detail, using \cite[Lemma 3 d)]{B09}, that $D$ is cleaving.
First of all we assume that there is a ray $\rho$ with $\rho\tilde{\beta}=\alpha^{t-1}$. Then we get $0\neq\alpha^t=\alpha\rho\tilde{\beta}=\beta_1\tilde{\beta}$, whence $\alpha\rho=\beta_1$ by the cancellation law. This contradicts the fact that $\beta_1$ is an arrow. In a similar way it can be shown that $\rho\alpha^{t-1}=\tilde{\beta},\ \rho v=\beta_1 w$ and $\rho\beta_1 w=v$ are impossible.\\
The following four cases are left to exclude.
\begin{enumerate}[i)]
\item $\alpha^{t-1}\rho=\beta_1 w$: Left multiplication with $\alpha$ gives us $\alpha^t\rho=\alpha\beta_1 w\neq 0$. Hence there is a non-deep contour $(\alpha^{t-1}\rho_1\ldots \rho_k,\beta_1 w_1\ldots w_l)$ in $\overrightarrow{A}$. Here $\rho=\rho_1\ldots\rho_k$ resp. $w=w_1\ldots w_l$ is a product of irreducible rays (arrows). Since the arrow $\beta_1$ is in the contour, the cycle $\beta_1 \tilde{\beta}$ and the loop $\alpha$ belong to the contour. Hence it can only be a penny-farthing by the structure theorem for non-deep contours \cite[6.4]{BGRS85}. But this case is excluded in the current section.
\item $\tilde{\beta}\rho=v$: We argue as before and deduce $\beta_1 \tilde{\beta}\rho=\beta_1 v=\alpha^t\rho=\alpha\beta_1 w\neq 0$. Hence there is a non-deep contour $(\alpha^{t-1}\rho_1\ldots \rho_k,\beta_1 w_1\ldots w_l)$ leading again to a contradiction.
\item $\beta_1 w\rho=\alpha^{t-1}$: Since $t-1<t$ we have a contradiction to the minimality of $t$.
\item $v\rho=\tilde{\beta}$: Then $\beta_1 v\rho=\beta_1 \tilde{\beta}=\alpha^t=\alpha\beta_1 v\rho\neq 0$. Using the cancellation law we get $\alpha^{t-1}=\beta_1 v\rho$ a contradiction as before.
\end{enumerate}
\end{proof}

\begin{lemma}\label{L3.5}
If $t\geq 3$ \textbf{and} $\mathcal{L}\nsubseteq\{\alpha^3,\alpha^2\beta_1 \}$, then $\alpha^2\beta_1 =0$.\\
\end{lemma}
\begin{proof}
If $\alpha^2\beta_1 \neq0$, then
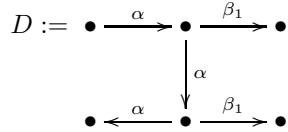
\[D:=\xymatrix@!=1pc{
\bullet \ar[r]^\alpha & \bullet \ar[d]^\alpha \ar[r]^{\beta_1}  & \bullet\\
\bullet & \bullet \ar[l]_\alpha \ar[r]^{\beta_1}  & \bullet
}\]
is a cleaving diagram of Euclidian type $\widetilde{D}_5$ in $\overrightarrow{A}$. It is cleaving since:
\begin{enumerate}[i)]
\item $\alpha^2=\beta_1\rho\neq 0$ contradicts the choice of $t\geq 3$.
\item $\alpha\beta_1=\beta_1\rho\neq 0$ contradicts Lemma \ref{L3.4}.
\end{enumerate}
  It is also cleaving in $\overrightarrow{A}/\eta$ for $\eta\in \mathcal{L}\setminus \{\alpha^3,\alpha^2\beta_1 \}\neq\emptyset$ contradicting the mildness of $A$.
\end{proof}

\begin{lemma}\label{L3.6}
If $\langle \alpha^2 \rangle \cap\langle \alpha\beta_1 \rangle =0=\langle \beta_1 \rangle \cap\langle \alpha\beta_1 \rangle$, then $\langle \alpha^2,\beta_1 \rangle \cap\langle \alpha\beta_1 \rangle =0$.
\end{lemma}
\begin{proof}
Let $\alpha^2 u+\beta_1 v=\alpha\beta_1 w\neq 0$ be an element in $\langle \alpha^2,\beta_1 \rangle \cap\langle \alpha\beta_1 \rangle $. By Lemma \ref{L1.9} we can assume that $u,v,w$ are rays and the following two cases might occur:
\begin{enumerate}[i)]
\item $\beta_1 v=\alpha\beta_1 w\neq 0$: This is a contradiction since $\langle \beta_1 \rangle \cap\langle \alpha\beta_1 \rangle =0$.
\item $\alpha^2u= \alpha\beta_1 w\neq 0$: This is impossible because $\langle \alpha^2\rangle \cap\langle \alpha\beta_1 \rangle =0$.
\end{enumerate}
\end{proof}

\subsection{The case $|x^+|=2$}
\begin{lemma}\label{L6.6}
If $x^+=\{\alpha,\beta_1\}$ and $\mathcal{L}\subseteq \{\alpha^3, \alpha^2\beta_1\}$, then there exists an $\alpha$-filtration $\mathcal{F}$ of $P_x$ having finite projective dimension.
\end{lemma}
\begin{proof}We treat two cases:
\begin{enumerate}[i)]
\item $\alpha\beta_1=0$: Then for $\langle\alpha^k\rangle$ with $k\geq 1$ only $S_x$ is possible as a composition factor; hence $\pd_{\Ld}\langle\alpha^k\rangle<\infty$. Thus $P_x\supset \langle\alpha\rangle\supset \langle\alpha^2\rangle\supset \langle\alpha^3\rangle\supset 0$ is the wanted $\alpha$-filtration.
\item $\alpha\beta_1\neq 0$: Since $\alpha^3$ and $\alpha^2\beta_1$ are the only morphisms in $\overrightarrow{A}$ which can be long, we have $t=3$, $0\neq\alpha^3\in\mathcal{L}$, $\langle\alpha\beta_1\rangle=\bk \alpha\beta_1\cong S_y$ and $\langle\alpha^2\beta_1\rangle\in\{\bk \alpha^2\beta_1,\ 0\}$.\\
    Now we show that $\langle \alpha^2\rangle \cap\langle \alpha\beta_1 \rangle =0$. If there are rays $v=v_1\ldots v_s,\ w \in \overrightarrow{A}$ with irreducible $v_i,\ i=1\ldots ,s$ such that $\alpha^2v=\alpha\beta_1 w\neq 0$, then $s>0$ because $s=0$ would contradict the irreducibility of $\alpha$. Therefore $v_1=\alpha$ or $v_1=\beta_1 $.
    \begin{itemize}
    \item[$\bullet$] If $v_1=\alpha$, then $v'=v_2\ldots v_s=id$ since $\alpha^3$ is long and $0\neq \alpha^2v=\alpha^3v'$. Hence $0\neq \alpha^3=\alpha^2v=\alpha\beta_1 w$ and $\alpha^2=\beta_1 w$ contradicts the minimality of $t$.
    \item[$\bullet$] If $v_1=\beta_1$, then $0\neq \alpha^2v=\alpha^2\beta_1 v'=\alpha\beta_1 w$; hence $0\neq \alpha\beta_1 v'=\beta_1 w\in \langle \beta_1 \rangle \cap\langle \alpha\beta_1 \rangle =0$.
    \end{itemize}
    Since $\langle\beta_1\rangle \cap\langle \alpha\beta_1 \rangle =0=\langle \alpha^2\rangle \cap\langle\alpha\beta_1\rangle$, we deduce $\langle \beta_1 ,\alpha^2, \alpha\beta_1 \rangle=\langle \beta_1 ,\alpha^2\rangle \oplus \langle \alpha\beta_1 \rangle$ by Lemma \ref{L3.6}. Therefore the graph of $P_x$ has the following shape:
    
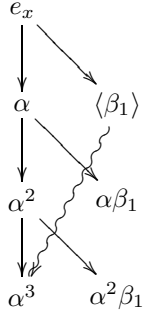
\[\xymatrix@!=1pc{
    e_x \ar[dr] \ar[d] & \\
    \alpha \ar[d] \ar[dr] & \langle \beta_1  \rangle\ar@{~>}[ddl] \\
    \alpha^2 \ar[d] \ar[dr]& \alpha\beta_1  \\
    \alpha^3 & \alpha^2\beta_1
    }.\]
    Here $\langle\beta_1\rangle$ stands for the graph of the submodule $\langle\beta_1\rangle$ which is not known explicitly.
    Consider the module $M$ defined by the following exact sequence: \[ 0\to \langle \beta_1 ,\alpha^2,\alpha\beta_1 \rangle \to P_x\to M\to 0\]
    Then $\pd_{\Ld} M<\infty$ since $M$ is filtered by $S_x$ and $\pd_{\Ld}(\langle \beta_1 ,\alpha^2\rangle\oplus\langle\alpha\beta_1 \rangle)=\pd_{\Ld}\langle \beta_1 ,\alpha^2,\alpha\beta_1 \rangle<\infty$. Thus $\pd_{\Ld} (\langle \alpha\beta_1 \rangle \cong S_y)$ is finite too and the wanted $\alpha$-filtration is $P_x \supset \langle\alpha\rangle \supset \langle\alpha^2\rangle \supset \langle\alpha^3\rangle\supset 0$.
\end{enumerate}
\end{proof}

\begin{lemma}\label{L3.8}
If $x^+=\{\alpha,\beta_1\}$, $t\geq 3$ \textbf{and} $\mathcal{L}\nsubseteq\{\alpha^3,\alpha^2\beta_1 \}$, then $\alpha^2\rho=0$ for all rays $\rho\notin \{e_x,\alpha,\ldots, \alpha^{t-2}\}$. Moreover, $\langle \alpha^2\rangle \cap \langle \alpha\beta_1 \rangle =0$.
\end{lemma}
\begin{proof}
Let $\rho\in \overrightarrow{A}$ with $\alpha^2\rho\neq 0$ be written as a composition of irreducible rays $\rho=\rho_1\ldots\rho_s$. Then the following two cases are possible:
\begin{enumerate}[i)]
\item $\rho=\alpha^s$: Since $0\neq \alpha^2\rho=\alpha^{2+s}$ and $\alpha^{t+1}=0$ we have $s\leq t-2$ and $\rho=\alpha^s\in \{e_x,\alpha,\ldots, \alpha^{t-2}\}$.
\item There exists a minimal $1\leq i\leq s$ such that $\rho_{i}\neq \alpha$. Since $x^+=\{\alpha,\beta_1\}$, we have $\rho_i=\beta_1$ and $0\neq\alpha^2\rho=\alpha^{2+i-1}\beta_1\rho_{i+1}\ldots\rho_s=0$ by Lemma \ref{L3.5}.
\end{enumerate}
If $0\neq\alpha^2 v=\alpha\beta_1 w$, then $v=\alpha^s$ with $0\leq s\leq t-2$. Hence $0=\alpha^2 v=\alpha^{s+2}=\alpha\beta_1 w$ and $\alpha^{s+1}=\beta_1 w$ by cancellation law. This contradicts the minimality of $t$.
\end{proof}

\begin{cor}\label{C3.9}
If $x^+=\{\alpha,\beta_1\}$, $t\geq 3$ \textbf{and} $\mathcal{L}\nsubseteq\{\alpha^3,\alpha^2\beta_1 \}$, then $\langle \alpha^2,\beta_1 \rangle \cap\langle \alpha\beta_1 \rangle =0$.
\end{cor}
\begin{proof}
The claim is trivial using Lemmas \ref{L3.4}, \ref{L3.6} and \ref{L3.8}.
\end{proof}

\begin{prop}\label{P3.10}
If $x^+=\{\alpha,\beta_1 \}$, then there exists an $\alpha$-filtration $\mathcal{F}$ of $P_x$ having finite projective dimension.
\end{prop}
\begin{proof} If $\mathcal{L}\subseteq\{\alpha^3,\alpha^2\beta_1 \}$, then the claim is the statement of Lemma \ref{L6.6}. If $\mathcal{L}\nsubseteq\{\alpha^3,\alpha^2\beta_1 \}$, then we consider the value of $t$:
\begin{enumerate}[i)]
\item $t=2$: Then the graph of $P_x$ has the following shape:
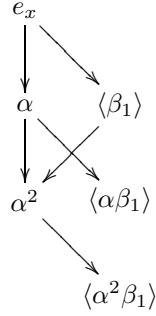
\[\xymatrix@!=1pc{
e_x \ar[dr] \ar[d] & \\
\alpha \ar[d] \ar[dr] & \langle \beta_1 \rangle  \ar[dl] \\
\alpha^2 \ar[dr]& \langle \alpha\beta_1 \rangle  \\
  & \langle \alpha^2\beta_1 \rangle
}.\]
Let a subquotient $M$ of $P_x$ be defined by the following exact sequence: \[0\to \langle \beta_1 ,\alpha\beta_1 \rangle \to P_x\to M\to 0\] Then
$M$ and $\langle \beta_1 ,\alpha\beta_1 \rangle $ have finite projective dimension in $\rmod\Ld$. By Lemma \ref{L3.4} we have  $\langle \beta_1 ,\alpha\beta_1 \rangle =\langle \beta_1 \rangle \oplus \langle \alpha\beta_1 \rangle $; hence $\pd_{\Ld} \langle \beta_1 \rangle$ and $\pd_{\Ld} \langle \alpha\beta_1 \rangle$ are both finite.\\
Let $K$ be the kernel of the epimorphism $\lambda_\alpha: \langle \beta_1 \rangle \to\langle \alpha\beta_1 \rangle ,\ \lambda_\alpha(\rho)=\alpha\rho$. Then $\pd_{\Ld} K< \infty$ and for the $\alpha$-filtration $\mathcal{F}$ we take the following: $P_x\supset \langle \alpha,\beta_1 \rangle \supset \langle \beta_1 \rangle \oplus \langle \alpha\beta_1 \rangle \supset \langle \alpha\beta_1 \rangle \oplus K\supset K\supset 0$.
\item $t\geq 3$: Consider the following exact sequences:
\[ 0\to \langle \alpha,\beta_1 \rangle \to P_x\to S_x\to 0\]
\[ 0\to \langle \alpha^2,\beta_1 ,\alpha\beta_1 \rangle \to \langle \alpha,\beta_1 \rangle \to S_x\to 0 \]
Hence $\pd_{\Ld} \langle \alpha,\beta_1 \rangle $ and $\pd_{\Ld} \langle \alpha^2,\beta_1 ,\alpha\beta_1 \rangle $ are finite. By Corollary \ref{C3.9} $\langle \alpha^2,\beta_1 ,\alpha\beta_1 \rangle =\langle \alpha^2,\beta_1 \rangle \oplus \langle \alpha\beta_1 \rangle $, that means $\pd_{\Ld} \langle \alpha\beta_1 \rangle $ is finite too. With Lemma \ref{L3.8} it is easily seen that for $2\leq k\leq t$ the module $\langle \alpha^k\rangle $ is a uniserial module with $S_x$ as the only composition factor. Hence $\pd_{\Ld} \langle \alpha^k\rangle $ is finite for $2\leq k\leq t$. Thereby we have the wanted $\alpha$-filtration \[P_x\supset \langle \alpha,\beta_1 \rangle \supset \langle \alpha^2\rangle \oplus \langle \alpha\beta_1 \rangle \supset \langle \alpha^3\rangle \supset \langle \alpha^4\rangle \supset\ldots \supset \langle \alpha^t\rangle \supset 0.\]
\end{enumerate}
\end{proof}

\subsection{The case $|x^+|=3$}
With previous notations $x^+=\{\alpha,\beta_1 ,\gamma \},\ (\alpha^t, \beta_1 \beta_2\ldots \beta_r)$ is a contour in $\overrightarrow{A}$, $t\geq 2$, $\alpha^{t+1}=0$, $\tilde{\beta}:=\beta_2\ldots \beta_r$ and $\mu=\nu\alpha^t \nu'$ is a long morphism in $\overrightarrow{A}$.

The $\alpha$-filtrations will be constructed depending on the set $\mathcal{L}$ of long morphisms in $\overrightarrow{A}$. The case $\mathcal{L}\subseteq\{ \alpha^2, \alpha\beta_1 , \alpha\gamma \}$ is treated in Lemma \ref{L3.14}, the case $\mathcal{L}\subseteq\{ \alpha^t, \alpha^2\beta_1\}$ in \ref{L3.16} and the remaining case in \ref{P3.20}.

But first, we derive some technical results.

\begin{lemma}\label{L3.11} If $r=2$ and $\delta :z' \to z$ is an arrow in $\sQ_A$ ending in $z=e(\gamma)$, then $\delta=\gamma$.
\end{lemma}
\begin{proof}
Assume to the contrary that $\gamma\neq\delta :z' \to z$, then there is no arrow $\beta_1\neq \varepsilon :y'\to y$ in $\sQ_{\Ld}$. If there is such an arrow, then by the definition of a neighborhood $\varepsilon$ belongs to $\sQ_A$. This arrow induces an irreducible ray $\beta_1\neq\varepsilon :y' \to y$ in $\overrightarrow{A}$ and
\[D:=\xymatrix@!=1pc{
\bullet \ar[r]^\delta & \bullet & \bullet \ar[d]^\alpha \ar[l]_\gamma  \ar[r]^{\beta_1}  & \bullet & \bullet \ar[l]_{\varepsilon}\\
& & \bullet & & \\
& & \bullet \ar[u]_{\beta_2} & &
}\] is a cleaving diagram in $\overrightarrow{A}/\mu$ of Euclidian type $\widetilde{E}_6$.\\
In a similar way an arrow $\alpha,\beta_2\neq \varepsilon:x'\to x$ in $\sQ_{\Ld}$ leads to a cleaving diagram of type $\widetilde{D}_5$ in $\overrightarrow{A}/\mu$.
Hence the full subcategory $B$ of $\Ld$ supported by the points $x,\ y$ is a convex subcategory of $\Ld$. Therefore the projective dimensions of $S_x$ is finite in $\rmod B$ since it is finite in $\rmod\Ld$. But in $B$ we have $x^+=\{\alpha,\beta_1 \}$, whence we can apply Proposition \ref{P3.10} together with \ref{lenzingsresult} to get the contradiction that $\alpha$ is not a loop.
\end{proof}

\begin{lemma}\label{L3.12}
If $\alpha\gamma \neq 0$, then $\beta_1 v\neq \alpha\gamma \neq \gamma w$ for all rays $v,w\in \overrightarrow{A}$.
\end{lemma}
\begin{proof}
\begin{itemize}
\item[i)] Assume that there exists a ray $v\in \overrightarrow{A}$ such that $\beta_1 v=\alpha\gamma \neq 0$. Then
\[D:=
\xymatrix@!=1pc{
\bullet \ar[d]_{\gamma} \ar[drr]^(.2){\alpha^{t-1}} & & \bullet \ar[d]^{\tilde{\beta}} \ar[dll]_(.2){v} \\
\bullet & & \bullet }
\]
is a cleaving diagram of Euclidian type $\widetilde{A}_3$ in $\overrightarrow{A}/\mu$.
\begin{itemize}
\item[$\bullet$] For $\gamma \rho=\alpha^{t-1}$ or $v\rho=\tilde{\beta}$ we have $\alpha\gamma \rho=\beta_1 v\rho=\beta_1 \tilde{\beta}=\alpha^t\neq 0$. Thus $\alpha^{t-1}=\gamma \rho$ contradicts the choice of $t$.
\item[$\bullet$] If $\alpha^{t-1}\rho=\gamma $ or $\tilde{\beta}\rho=v$, then $\alpha^t\rho=\beta_1 \tilde{\beta}\rho=\beta_1 v=\alpha\gamma \neq 0$. Then $\alpha^{t-1}\rho=\gamma $ contradicts the irreducibility of $\gamma$.
\end{itemize}
\item[ii)] Assume that there exists a ray $w=w_1\ldots w_s : z\rightsquigarrow z \in \overrightarrow{A}$ with irreducible $w_i$ such that $\gamma w=\alpha\gamma \neq 0$.
\begin{itemize}
\item[$r=2$:] Since $w_s$ is an irreducible ray ending in $z$, $w_s=\gamma $ by Lemma \ref{L3.11}. Thus we get a contradiction $\gamma w_1\ldots w_{s-1}=\alpha$.
\item[$r\geq 3$:] We look at the value of $s$. If $s=1$, then $w=w_1$ is a loop and
\[D:=\xymatrix@!=1pc{
\bullet \ar[r]^{w_1=w} & \bullet \ar@/_1pc/@{..>}[ddr]_{(1)} & \bullet \ar[l]_\gamma  \ar[d]^{\beta_1}  \ar[r]^\alpha & \bullet \ar@/^1pc/@{..>}[ddl]^{(2)} & \bullet \ar[l]_{\beta_r}\\
 &  & \bullet \ar[d]^{\beta_2} & & \\
 & & \bullet & & }\] is a cleaving diagram in $\overrightarrow{A}/\mu$.\\
  If $s\geq 2$, then
\[D:=\xymatrix@!=1pc{
\bullet \ar[r]^{w_{s-1}} \ar@/_2pc/@{..>}[rrr]_{(3)} & \bullet \ar[r]^{w_{s}} & \bullet & \bullet \ar[l]_\gamma  \ar[d]^{\beta_1}  \ar[r]^\alpha & \bullet & \bullet \ar[l]_{\beta_r} & \bullet \ar[l]_{\beta_{r-1}} \ar@/^2pc/@{..>}[lll]^{(4)}\\
 &  &  & \bullet  & & & }\] is cleaving in $\overrightarrow{A}/\mu$.\\
We still have to show that not any morphisms indicated by the dotted lines make the diagrams commute.
 \begin{itemize}
 \item[(1):] $\gamma \rho=\beta_1 \beta_2$, with $\rho=\rho_1\ldots \rho_l$. If $\rho=w_1^l=w^l$, then $\beta_1 \beta_2=\gamma \rho=\gamma w^l=\alpha\gamma w^{l-1}$ and $\beta_1 \beta_2\ldots \beta_r=\alpha^t=\alpha\gamma w^{l-1}\beta_3\ldots \beta_r\neq 0$. Therefore $\alpha^{t-1}=\gamma w^{l-1}\beta_3\ldots \beta_r$ is a contradiction. If $\rho\neq w_1^l$, then one of the irreducible rays $\rho_i\neq w_1$ starts in $z$ and
\[D:=\xymatrix@!=1pc{
\bullet & \bullet \ar[l]_{\rho_i} \ar[r]^{w_{1}} & \bullet & \bullet \ar[l]_\gamma  \ar[d]^{\beta_1}  \ar[r]^\alpha & \bullet & \bullet \ar[l]_{\beta_r} & \bullet \ar[l]_{\beta_{r-1}} \ar@/^2pc/@{..>}[lll]^{(4)}\\
 &  &  & \bullet  & & & }\] is cleaving in $\overrightarrow{A}/\mu$.\\
 \item[(2):] If $\alpha\rho=\beta_1 \beta_2$, then $\alpha\rho \beta_3\ldots \beta_r=\beta_1 \beta_2\ldots \beta_r=\alpha^t\neq 0$ and $\alpha^{t-1}=\rho \beta_3\ldots \beta_r$ contradicts the minimality of $t$.
 \item[(3):] If $\rho \gamma =w_{s-1}w_s$, then $\gamma w_1\ldots w_{s-2}\rho \gamma =\gamma w=\alpha\gamma \neq 0$ and $\alpha=\gamma w_1\ldots w_{s-2}\rho$ contradicts the irreducibility of $\alpha$.
 \item[(4):] If $\rho \alpha=\beta_{r-1}\beta_r$, then $\beta_1 \beta_2\ldots \beta_{r-2}\rho \alpha=\beta_1 \beta_2\ldots \beta_r=\alpha^t\neq 0$ and $\alpha^{t-1}=\beta_1 \beta_2\ldots \beta_{r-2}\rho$ contradicts the minimality of $t$.
 \end{itemize}
\end{itemize}
\end{itemize}
\end{proof}

\begin{lemma}\label{L3.13}
If $t\geq 3$, then $\alpha\gamma =0$.
\end{lemma}
\begin{proof} Assume that $\alpha\gamma \neq 0$, then
\[D:=\xymatrix@!=1pc{
\bullet & \bullet \ar[l]_\gamma  \ar[r]^{\beta_1}  \ar[d]^\alpha & \bullet\\
\bullet & \bullet \ar[l]_\gamma  \ar[r]^\alpha & \bullet }\] is a cleaving diagram of Euclidian type in $\overrightarrow{A}/\mu$. It is cleaving since:
\begin{enumerate}[i)]
\item $\gamma \rho=\alpha\gamma$ or $\beta_1 \rho=\alpha\gamma $ contradicts Lemma \ref{L3.12},
\item $\gamma \rho=\alpha^2$ or $\beta_1 \rho=\alpha^2$ contradicts the minimality of $t\geq 3$.
\end{enumerate}
\end{proof}

\begin{lemma}\label{L3.15}
\begin{enumerate}[a)]
\item[]
\item If $\mathcal{L}\nsubseteq\{\alpha^2, \alpha\beta_1 , \alpha\gamma\}$, then $\alpha\beta_1 =0$ or $\alpha\gamma =0$.
\item If $\alpha^2\beta_1 \neq 0$, then $\gamma w\neq \alpha\beta_1 $ for all $w\in \overrightarrow{A}$.
\end{enumerate}
\end{lemma}
\begin{proof}
\begin{enumerate}[a)]
\item If $\alpha\beta_1 \neq 0$ and $\alpha\gamma \neq 0$, then
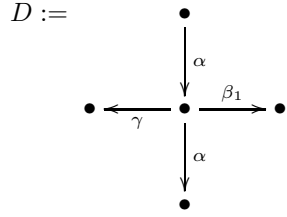
\[D:=\xymatrix@!=1pc{
& \bullet \ar[d]^\alpha  & \\
\bullet & \bullet \ar[l]^\gamma  \ar[r]^{\beta_1}  \ar[d]^\alpha & \bullet \\
 & \bullet &
}\] is a cleaving diagram of Euclidian type $\widetilde{D}_4$ in $\overrightarrow{A}$. It is still cleaving in $\overrightarrow{A}/\eta$ for $\eta\in \mathcal{L}\setminus \{ \alpha^2,\alpha\beta_1 ,\alpha\gamma \}\neq\emptyset$.
\item Since $\alpha^2\beta_1 \neq 0$, we have $\alpha\gamma =0$ by $a)$. But $\gamma w=\alpha\beta_1$ leads to the contradiction $0\neq \alpha^2\beta_1 =\alpha\gamma w=0$.
\end{enumerate}
\end{proof}

\begin{lemma}\label{L3.17}
If $t=2$ \textbf{or} $\mathcal{L}\nsubseteq\{\alpha^t, \alpha^2\beta_1\}$, then:
\begin{enumerate}[a)]
\item $\alpha^2\beta_1 =0=\alpha^2\gamma $, $\alpha^2\rho=0$ for all rays $\rho\notin \{e_x,\alpha,\ldots, \alpha^{t-2}\}$.
\item $\langle \beta_1 \rangle \cap \langle \alpha\gamma \rangle =0$.
\item If $\langle\gamma\rangle \cap \langle \beta_1 \rangle =0$, then $\langle \gamma \rangle \cap \langle \alpha^2\rangle=0$.
\item $\langle \gamma \rangle \cap \langle \alpha^t\rangle =0$ or $\langle \gamma \rangle \cap \langle \alpha\beta_1 \rangle =0$.
\item $\langle \gamma \rangle \cap\langle \alpha\beta_1 \rangle =0$ or $\langle \gamma \rangle \cap \langle \beta_1 \rangle =0$.
\item $\langle \alpha\beta_1 \rangle \cap \langle \alpha^2\rangle =0$ and $\langle \alpha\gamma \rangle \cap \langle \alpha^2\rangle =0$.
\end{enumerate}
\end{lemma}
\begin{proof}
\begin{itemize}
\item[a)] Consider the case $t=2$.
    \begin{itemize}
    \item[i)] If $\alpha^2\beta_1 \neq 0$, then $\beta_r\beta_1 \neq 0$ and
        
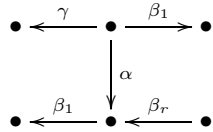
\[\xymatrix@!=1pc{
        \bullet & \bullet \ar[d]^\alpha \ar[l]_\gamma  \ar[r]^{\beta_1}  & \bullet\\
        \bullet & \bullet \ar[l]_{\beta_1}  & \bullet \ar[l]_{\beta_r}
        }\] is a cleaving diagram of Euclidian type $\widetilde{D}_5$ in $\overrightarrow{A}/\mu$. The diagram is cleaving because:
        \begin{itemize}
        \item[$\bullet$] $\beta_1 \rho=\alpha\beta_1 \neq 0$ is a contradiction of Lemma \ref{L3.4},
        \item[$\bullet$] $\gamma \rho=\alpha\beta_1 \neq 0$ contradicts Lemma \ref{L3.15} b).
        \end{itemize}
    \item[ii)] If $\alpha^2\gamma \neq 0$, then $\beta_r\gamma\neq 0$ and
        
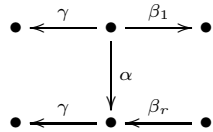
\[\xymatrix@!=1pc{
        \bullet & \bullet \ar[d]^\alpha \ar[l]_\gamma  \ar[r]^{\beta_1}  & \bullet\\
        \bullet & \bullet \ar[l]_\gamma  & \bullet \ar[l]_{\beta_r}
        }\] is a cleaving diagram in $\overrightarrow{A}/\mu$. It is cleaving since $\beta_1 \rho=\alpha\gamma $ resp. $\gamma \rho=\alpha\gamma$ contradicts Lemma \ref{L3.12}.
    \end{itemize}
    In the case $t\geq 3$, $\alpha^2\gamma =0$ by Lemma \ref{L3.13}. If $t=3$, then $\mathcal{L}\nsubseteq\{\alpha^3, \alpha^2\beta_1\}$ by assumption. If $t>3$, then $\mu=\nu\alpha^t\nu'\in \mathcal{L}\setminus\{\alpha^3, \alpha^2\beta_1\}$. Hence $\alpha^2\beta_1=0$ by Lemma \ref{L3.5} in both cases.
\item[b)] If $v,w$ are rays in $\overrightarrow{A}$ such that $\beta_1 v=\alpha\gamma w\neq 0$, then the diagram
    \[D:=
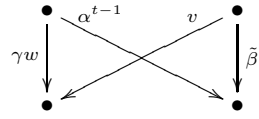
\xymatrix@!=1pc{
    \bullet \ar[d]_{\gamma w} \ar[drr]^(.2){\alpha^{t-1}} & & \bullet \ar[d]^{\tilde{\beta}} \ar[dll]_(.2){v} \\
    \bullet & & \bullet }\] is a cleaving diagram in $\overrightarrow{A}/\mu$.
    \begin{enumerate}[i)]
        \item If $\gamma w\rho=\alpha^{t-1}$ or $v\rho=\tilde{\beta}$, then $\beta_1 v\rho=\beta_1 \tilde{\beta}=\alpha^t=\alpha\gamma w\rho\neq 0$. Hence $\gamma w\rho=\alpha^{t-1}$ contradicts the minimality of $t$.
        \item If $\alpha^{t-1}\rho=\gamma w$ or $\tilde{\beta}\rho=v$, then $0\neq\beta_1 v=\beta_1 \tilde{\beta}\rho=\alpha\gamma w=\alpha^t\rho=0$ by a).
    \end{enumerate}
\item[c)]  Let $v,w$ be rays such that $\gamma v=\alpha^2w\neq 0$. By $a)$ we have $w=\alpha^k$ with $0\leq k\leq t-2$, that means $\gamma v=\alpha^{2+k}$. Since $t$ is minimal, we have $t=2+k$ and
    $0\neq \gamma v=\alpha^t=\beta_1 \tilde{\beta}\in\langle \gamma \rangle \cap \langle \beta_1 \rangle =0$.

\item[d)] Let $v,w,v',w'$ be rays in $\overrightarrow{A}$ such that $\gamma w=\alpha^tv\neq 0$ and $\gamma w'=\alpha\beta_1 v'\neq 0$. Then
        \[D:=
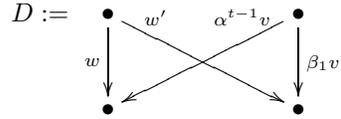
\xymatrix@!=1pc{
        \bullet \ar[d]_{w} \ar[drr]^(.2){w'} & & \bullet \ar[d]^{\beta_1 v'} \ar[dll]_(.2){\alpha^{t-1}v} \\
        \bullet & & \bullet }\] is a cleaving diagram in $\overrightarrow{A}/\mu$.
        \begin{enumerate}[i)]
            \item If $w\rho=w'$ or $\alpha^{t-1}v\rho=\beta_1 v'$, then $\gamma w\rho=\gamma w'=\alpha^tv\rho=\alpha\beta_1 v'\neq 0$. Hence there is a non-deep contour $(\alpha^{t-1}v_1\ldots v_k\rho_1\ldots\rho_l,\beta_1 v'_1\ldots v'_s)$ in $\overrightarrow{A}$ which can only be a penny-farthing by the structure theorem for non-deep contours. But this case is excluded in the current section.
            \item If $w'\rho=w$ or $\beta_1 v'\rho=\alpha^{t-1}v$, then $\gamma w'\rho=\gamma w=\alpha\beta_1 v'\rho=\alpha^tv\neq 0$. Again, we have a non-deep contour $(\alpha^{t-1}v_1\ldots v_k, \beta_1 v'_1\ldots v'_l\rho_1\ldots\rho_s)$ which leads to a contradiction as before.
        \end{enumerate}

\item[e)] Let $v,w,v',w'$ be rays such that $\beta_1 v=\gamma w\neq 0$ and $\alpha\beta_1 v'=\gamma w'\neq 0$. Then
    
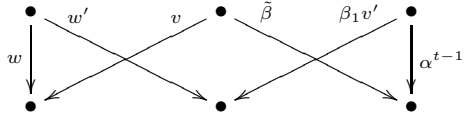
\[\xymatrix@!=1pc{
    \bullet \ar[d]_{w} \ar[drr]^(.2){w'} & & \bullet \ar[drr]^(.2){\tilde{\beta}} \ar[dll]_(.2){v} & & \bullet \ar[d]^{\alpha^{t-1}} \ar[dll]_(.2){\beta_1 v'}\\
    \bullet & & \bullet & & \bullet}\] is a cleaving diagram in $\overrightarrow{A}/\mu$.
        \begin{enumerate}[i)]
            \item If $w\rho=w'$, we get the contradiction $0\neq \gamma w\rho=\gamma w'=\beta_1 v\rho=\alpha\beta_1 v'\in \langle \beta_1 \rangle \cap \langle \alpha\beta_1 \rangle =0$.
            \item If $w'\rho=w$, then $0\neq \gamma w'\rho=\gamma w=\alpha\beta_1 v'\rho=\beta_1 v\in \langle \beta_1 \rangle \cap \langle \alpha\beta_1 \rangle =0$.
            \item If $v\rho=\tilde{\beta}$, then $0\neq \beta_1 v\rho=\beta_1 \tilde{\beta}=\gamma w\rho=\alpha^t\in \langle \gamma \rangle \cap \langle \alpha^t\rangle =0$ by d).
            \item If $\tilde{\beta}\rho=v$, then $0\neq \beta_1 \tilde{\beta}\rho=\beta_1 v=\alpha^t\rho=\gamma w\in \langle \gamma \rangle \cap \langle \alpha^t\rangle =0$ by d).
            \item If $\alpha^{t-1}\rho=\beta_1 v'$, then $0\neq \alpha^t\rho=\alpha\beta_1 v'=\gamma w'\in \langle \gamma \rangle \cap \langle \alpha^t\rangle =0$ by d).
            \item The case $\beta_1 v'\rho=\alpha^{t-1}$ contradicts the minimality of $t$.
        \end{enumerate}
\item[f)] If $v,w$ are rays in $\overrightarrow{A}$ such that $\alpha\beta_1 v=\alpha^2w\neq 0$ resp. $\alpha\gamma v=\alpha^2w\neq 0$, then $w=\alpha^k$ with $0\leq k\leq t-2$ and $\beta_1 v=\alpha^{1+k}$ resp. $\gamma v=\alpha^{1+k}$. Since $t$ is minimal, we get the contradiction $t=1+k<t$.
\end{itemize}
\end{proof}

\begin{lemma}\label{L6.17}
If $\mathcal{L}\nsubseteq\{\alpha^2,\alpha\beta_1,\alpha\gamma\}$, then $\langle \gamma \rangle \cap\langle \alpha\gamma \rangle =0$.
\end{lemma}
\begin{proof}
In the case $t\geq 3$, the claim is trivial since $\alpha\gamma =0$ by \ref{L3.13}.\par
Consider the case $t=2$. Assume that there exist rays $v,w$ in $\overrightarrow{A}$ such that $\gamma v=\alpha\gamma w\neq 0$. First of all, we deduce that $w\neq id$ by Lemma \ref{L3.12} and $v\neq id$ since $\gamma$ is an arrow. Therefore we can write $v=v_1\ldots v_s$, $w=w_1\ldots ,w_q$ with irreducible rays $v_i, w_j\in \overrightarrow{A}$. Consider the value of $q$:
\begin{enumerate}[a)]
\item If $q=1$, then the diagram
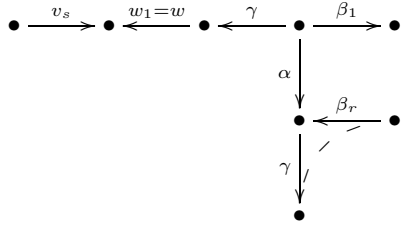
\[\xymatrix@!=1pc{
\bullet \ar[r]^{v_s} & \bullet & \bullet \ar[l]_{w_1=w} & \bullet \ar[l]_\gamma  \ar[r]^{\beta_1}  \ar[d]_\alpha & \bullet \\
 & & & \bullet \ar[d]_\gamma  & \bullet \ar[l]_{\beta_r} \ar@{--}@/_1pc/[dl]\\
 & & & \bullet &
}\] is a cleaving diagram of Euclidian type $\widetilde{E}_7$ in $\overrightarrow{A}/\mu$ (see \cite[10.7]{GR92}).
\item If $q\geq 2$, then the diagram
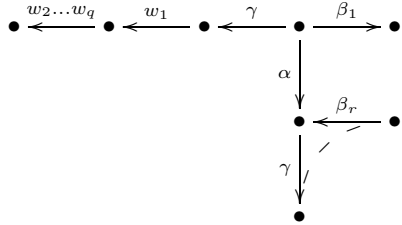
\[\xymatrix@!=1pc{
\bullet & \bullet \ar[l]_{w_2\ldots w_q} & \bullet \ar[l]_{w_1} & \bullet \ar[l]_\gamma  \ar[r]^{\beta_1}  \ar[d]_\alpha & \bullet \\
 & & & \bullet \ar[d]_\gamma  & \bullet \ar[l]_{\beta_r} \ar@{--}@/_1pc/[dl]\\
 & & & \bullet &
}\] is cleaving in $\overrightarrow{A}/\mu$.
\end{enumerate}
The diagrams are cleaving because:
\begin{enumerate}[i)]
\item $\alpha\rho=\gamma w\neq 0$: Then $0\neq\alpha\gamma w=\alpha^2\rho=0$ by Lemma \ref{L3.17} a).
\item $\gamma \rho=\alpha\gamma \neq 0$ contradicts Lemma \ref{L3.12}.
\item $\beta_1 \rho=\gamma w\neq 0$: Then $0\neq\alpha\gamma w=\alpha\beta_1 \rho=0$ since $\alpha\beta_1 =0$ by Lemma \ref{L3.15}.
\item $\rho v_s=\gamma w\neq 0$: Then $\alpha\rho v_s=\alpha\gamma w\neq 0$. If $\rho=\beta_1 \rho'$, then $0=\alpha\beta_1 \rho' v_s=\alpha\gamma w\neq 0$. If $\rho=\gamma \rho'$, then $\alpha\gamma \rho' v_s=\alpha\gamma w$ and $w_1=w=\rho' v_s$. Hence $\rho'=id$ and $v_s=w_1$. Therefore $0\neq\gamma v=\gamma v_1\ldots v_{s-1}w_1=\alpha\gamma w_1$ and $\gamma v_1\ldots v_{s-1}=\alpha\gamma $ contradicting Lemma \ref{L3.12}. If $\rho=\alpha\rho'$, then $0\neq\alpha\gamma w=\alpha^2\rho'v_s=0$ by Lemma \ref{L3.17} a).
\item $\beta_1 \rho=\alpha\gamma \neq 0$ contradicts Lemma \ref{L3.12}.
\end{enumerate}
\end{proof}

\begin{lemma}\label{L3.20}
Let $\mathcal{L}\nsubseteq\{\alpha^t,\alpha^2\beta_1\}$ and $\mathcal{L}\nsubseteq\{\alpha^2,\alpha\beta_1,\alpha\gamma\}$.
\begin{itemize}
\item[a)] If $\langle \alpha\gamma \rangle =0=\langle \gamma \rangle \cap \langle \alpha\beta_1 \rangle $, then $\langle \beta_1 ,\gamma ,\alpha^2\rangle \cap \langle \alpha\beta_1 \rangle =0$.
\item[b)] If $\langle \alpha\gamma \rangle =0=\langle \gamma \rangle \cap \langle \beta_1 \rangle $, then $\langle \beta_1 ,\alpha^2\rangle \cap \langle \gamma ,\alpha\beta_1 \rangle =0$.
\item[c)] If $\langle \alpha\beta_1 \rangle =0$, then $\langle \beta_1 ,\gamma ,\alpha^2\rangle \cap \langle \alpha\gamma \rangle =0$.
\end{itemize}
\end{lemma}
\begin{proof} We only prove b); the other cases are proven analogously. Let $v,v',w,w'\in A$ be such that $\beta_1 v+\alpha^2v'=\gamma w+\alpha\beta_1 w'\neq 0$. That means we have rays $v_i , w_j\in \overrightarrow{A}$, numbers $\lambda_i,\mu_j\in\bk$ and integers $s_1,s_2\geq 0,\ n_1,n_2\geq 1$ such that
\[\sum_{i=1}^{s_1}\ld_i \beta_1 v_i + \sum_{i=s_1+1}^{n_1}\ld_i \alpha^2v_i=\sum_{j=1}^{s_2}\mu_j \gamma w_j + \sum_{j=s_2+1}^{n_2}\mu_j \alpha\beta_1 w_j\] and $\beta_1 v_i\neq \beta_1 v_j,\ \alpha^2v_i\neq \alpha^2v_j,\ \gamma w_i\neq \gamma w_j,\ \alpha\beta_1 w_i\neq \alpha\beta_1 w_j$ for $i\neq j$. Without loss of generality we can assume that all $\ld_i,\mu_j$ are non-zero, that $\beta_1 v_i\neq \alpha^2v_j$ for $i=1\ldots s_1,\ j=s_1+1\ldots n_1$ and $\gamma w_i\neq \alpha\beta_1 w_j$ for $i=1\ldots s_2,\ j=s_2+1\ldots n_2$.
Then by Lemma \ref{L1.9} we have $n_1=n_2$ and there exists a permutation $\pi$ such that $\beta_1 v_i=\gamma w_{\pi(i)}\in \langle \beta_1 \rangle \cap \langle \gamma \rangle =0$ or $\beta_1 v_i=\alpha\beta_1 w_{\pi(i)}\in \langle \beta_1 \rangle \cap \langle \alpha\beta_1 \rangle =0$ by Lemma \ref{L3.4}. Hence $s_1=0$. Moreover, by Lemma \ref{L3.17} we have $\alpha^2v_i=\gamma w_{\pi(i)}\in \langle \alpha^2\rangle \cap \langle \gamma \rangle=0$ or $\alpha^2v_i=\alpha\beta_1 w_{\pi(i)}\in \langle \alpha^2\rangle \cap \langle \alpha\beta_1 \rangle =0$; this is possible for $n_1-s_1=0$ only. Hence $n_1=0$, contradicting the choice of $n_1$.
\end{proof}

\begin{lemma}\label{L3.14}
If $\mathcal{L}\subseteq\{ \alpha^2, \alpha\beta_1 , \alpha\gamma \}$, then there exists an $\alpha$-filtration $\mathcal{F}$ of $P_x$ having finite projective dimension.
\end{lemma}
\begin{proof}
Since $\mathcal{L}\subseteq\{ \alpha^2, \alpha\beta_1 , \alpha\gamma \}$, $\mu=\alpha^2$ is long and $t=2$. Now it is easily seen that $\langle \alpha^2 \rangle =\bk \alpha^2\cong S_x$, $\langle \alpha\gamma \rangle =\bk \alpha\gamma $, $\langle \alpha\beta_1 \rangle =\bk \alpha\beta_1$ and $\langle \alpha \rangle$ has a $\bk$ basis $\{\alpha,\alpha^2,\alpha\beta_1,\alpha\gamma\}$. Using Lemma \ref{L3.4} and \ref{L3.12} we conclude $\langle \beta_1 \rangle \cap \langle \alpha\beta_1 \rangle =0$ and $\langle \gamma \rangle \cap \langle \alpha\gamma \rangle =0=\langle \beta_1 \rangle \cap \langle \alpha\gamma \rangle $.\\
By Lemma \ref{L3.17} d) $\langle \gamma \rangle \cap \langle \alpha^2\rangle =0$ or $\langle \gamma \rangle \cap \langle \alpha\beta_1 \rangle=0$.
Thus the graph of $P_x$ has one of the following shapes:
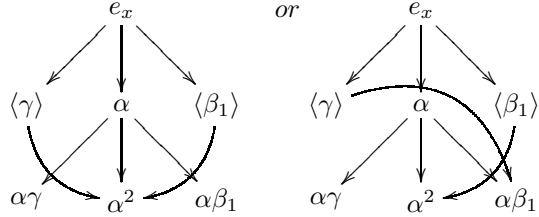
\[\xymatrix@!=1pc{
& e_x \ar[dr] \ar[dl] \ar[d] & \\
\langle \gamma  \rangle \ar@/_1pc/@{->}[dr] & \alpha \ar[d] \ar[dl] \ar[dr] & \langle \beta_1  \rangle \ar@/^1pc/@{->}[dl] \\
\alpha\gamma  & \alpha^2 & \alpha\beta_1
}\ \ \ or
\xymatrix@!=1pc{
& e_x \ar[dr] \ar[dl] \ar[d] & \\
\langle \gamma  \rangle  \ar@/^2pc/@{->}[drr] & \alpha \ar[d] \ar[dl] \ar[dr] & \langle \beta_1  \rangle \ar@/^1pc/@{->}[dl] \\
\alpha\gamma  & \alpha^2 & \alpha\beta_1
}.\]
In the first case we consider the following exact sequence:
\[0\to \langle \alpha^2  \rangle\to\langle \alpha,\beta_1,\gamma  \rangle\to \langle \alpha,\beta_1,\gamma  \rangle/\langle \alpha^2  \rangle \to 0\] Since $\langle \alpha \rangle$ has $\bk$ basis $\{\alpha,\alpha^2,\alpha\beta_1,\alpha\gamma  \rangle$ and $\mathcal{L}\subseteq\{ \alpha^2, \alpha\beta_1 , \alpha\gamma \}$ we have $\langle \alpha,\beta_1,\gamma  \rangle/\langle \alpha^2  \rangle=\langle \alpha \rangle/\langle \alpha^2  \rangle\oplus \langle \beta_1,\gamma  \rangle/\langle \alpha^2  \rangle$. Hence $\pd_{\Ld}\langle\alpha\rangle<\infty$ and
$P_x \supset \langle\alpha\rangle \supset \langle\alpha^2\rangle \supset 0$ is the wanted filtration.\par
In the second case we have $\langle \alpha,\beta_1,\gamma  \rangle/\langle \alpha^2  \rangle=\langle \alpha,\gamma \rangle/\langle \alpha^2  \rangle\oplus \langle \beta_1  \rangle/\langle \alpha^2  \rangle$. Thus $\pd_{\Ld}\langle \alpha,\gamma  \rangle<\infty$. Now we consider
\[0\to \langle \beta_1,\gamma,\alpha\gamma  \rangle\to\langle \alpha,\beta_1,\gamma  \rangle\to S_x \to 0.\]
Since $\langle \beta_1,\gamma,\alpha\gamma  \rangle =\langle \beta_1,\gamma  \rangle\oplus \langle \alpha\gamma  \rangle$, we have $\pd_{\Ld}\langle \alpha\gamma  \rangle<\infty$ and  $P_x \supset \langle \alpha,\gamma \rangle  \supset \langle \alpha^2, \alpha\gamma \rangle  \supset 0$ is a suitable filtration.
\end{proof}

\begin{lemma}\label{L3.16}
If $\mathcal{L}\subseteq\{\alpha^t,\alpha^2\beta_1\}$, then there exists an $\alpha$-filtration $\mathcal{F}$ of $P_x$ having finite projective dimension.
\end{lemma}
\begin{proof} If $t=2$, then $\alpha^2\beta_1=0$ by Lemma \ref{L3.17} a). Hence $\mathcal{L}\subseteq\{\alpha^2\}$ and the filtration exists by Lemma \ref{L3.14}.\\
If $t\geq 3$, then $\alpha\gamma =0$ by Lemma \ref{L3.13}. From the assumption $\mathcal{L}\subseteq\{\alpha^t,\alpha^2\beta_1\}$ it is easily seen that $\langle\alpha\beta_1\rangle =\bk\alpha\beta_1$ and $\langle\alpha^2\beta_1\rangle =\bk\alpha^2\beta_1$.
\begin{enumerate}[i)]
\item If $\alpha^2\beta_1=0$, then $\alpha^t$ is the only long morphism in $\overrightarrow{A}$; hence $\alpha\beta_1=0$ and $\langle\alpha^k\rangle$, $k\geq 1$, is uniserial of finite projective dimension. Thus $P_x\supset \langle\alpha\rangle\supset \langle\alpha^2\rangle\supset\ldots\supset \langle\alpha^t\rangle\supset 0$ is a suitable $\alpha$-filtration.
\item If $\alpha^2\beta_1\neq 0$, then $\langle\alpha\beta_1\rangle =\bk \alpha\beta_1\cong S_y\cong\langle\alpha^2\beta_1\rangle$. By \ref{L3.4} and \ref{L3.15} b)\\ $\langle\beta_1\rangle\cap\langle\alpha\beta_1\rangle=0=\langle\gamma\rangle\cap\langle\alpha\beta_1\rangle$. Therefore the graph of $P_x$ has the following shape:
    
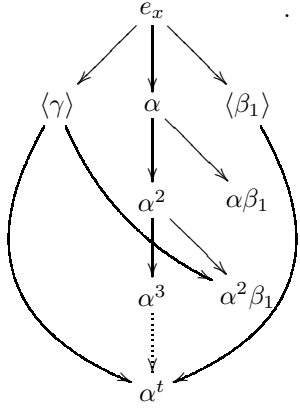
\[\xymatrix@!=1pc{
    & e_x \ar[dr] \ar[dl] \ar[d] & \\
    \langle \gamma \rangle \ar@/_3pc/@{->}[dddr] \ar@/_1pc/@{->}[ddrr] & \alpha \ar[d] \ar[dr] & \langle \beta_1 \rangle  \ar@/^3pc/@{->}[dddl] \\
     & \alpha^2 \ar[d] \ar[dr] & \alpha\beta_1  \\
     & \alpha^3 \ar@{..>}[d] & \alpha^2\beta_1  \\
     & \alpha^t &
    }.\]
    Moreover, $\langle\alpha\beta_1\rangle \cong S_y$ is a direct summand of the module $\langle\alpha^2,\beta_1,\gamma,\alpha\beta_1\rangle$, which has finite projective dimension. Since the modules $\langle\alpha\rangle, \langle\alpha^2\rangle, \ldots, \langle\alpha^t\rangle$ have $S_x$ and $S_y$ as the only composition factors, they are of finite projective dimension. Thus $P_x \supset \langle\alpha\rangle \supset \langle\alpha^2\rangle \supset \ldots \langle\alpha^t\rangle \supset 0$ is a suitable $\alpha$-filtration.
\end{enumerate}
\end{proof}

\begin{prop}\label{P3.20}
If $x^+=\{\alpha,\beta_1 ,\gamma \}$, then there exists an $\alpha$-filtration $\mathcal{F}$ of $P_x$ having finite projective dimension.
\end{prop}
\begin{proof}
By lemmata \ref{L3.14} and \ref{L3.16} we can assume that $\mathcal{L}\nsubseteq\{\alpha^t,\alpha^2\beta_1\}$ and $\mathcal{L}\nsubseteq\{\alpha^2,\alpha\beta_1 ,\alpha\gamma\}$. Then $\pd_{\Ld} \langle \alpha^k\rangle <\infty$ for $2\leq k\leq t$ since $\langle \alpha^k\rangle $ has only $S_x$ as a composition factor by \ref{L3.17} a). Moreover, $\pd_{\Ld} \langle \alpha,\beta_1 ,\gamma \rangle <\infty$ since it is the left hand term of the following exact sequence: \[ 0\to \langle\alpha,\beta_1,\gamma\rangle \to P_x \to S_x\to 0. \]
By Lemma \ref{L3.15} a) only the following two cases are possible:
\begin{itemize}
\item[i)] $\alpha\beta_1 =0$: Consider the following exact sequence:
\[ 0\to \langle  \beta_1 ,\gamma ,\alpha^2,\alpha\gamma \rangle\to \langle \alpha,\beta_1 ,\gamma \rangle \to S_x\to 0 .\]
Then $\pd_{\Ld} \langle \beta_1 ,\gamma ,\alpha^2,\alpha\gamma \rangle <\infty$. By \ref{L3.20} c) we have $\langle \beta_1 ,\gamma ,\alpha^2,\alpha\gamma \rangle =\langle \beta_1 ,\gamma ,\alpha^2\rangle \oplus \langle \alpha\gamma \rangle $; hence $\pd_{\Ld} \langle \alpha\gamma \rangle <\infty$. Therefore
$P_x \supset \langle \alpha,\beta_1 ,\gamma \rangle  \supset \langle \alpha^2\rangle \oplus \langle \alpha\gamma \rangle  \supset \langle \alpha^3\rangle  \supset\ldots \langle \alpha^t\rangle \supset 0$ is a suitable $\alpha$-filtration.

\item[ii)] $\alpha\gamma =0$: Then $\pd_{\Ld} \langle \beta_1 ,\gamma ,\alpha^2,\alpha\beta_1 \rangle <\infty$ since we have the exact sequence
\[ 0\to \langle \beta_1 ,\gamma ,\alpha^2,\alpha\beta_1 \rangle \to \langle \alpha,\beta_1 ,\gamma \rangle \to S_x\to 0 .\]
If $\langle \gamma \rangle \cap \langle \alpha\beta_1 \rangle =0$, then by \ref{L3.20} a) we have $\langle  \beta_1 ,\gamma ,\alpha^2,\alpha\beta_1 \rangle =\langle  \beta_1 ,\gamma ,\alpha^2\rangle \oplus \langle  \alpha \beta_1 \rangle $; hence $\pd_{\Ld} \langle \alpha\beta_1 \rangle <\infty$. Therefore
$P_x \supset \langle \alpha,\beta_1 ,\gamma \rangle  \supset \langle \alpha^2\rangle \oplus \langle \alpha \beta_1 \rangle  \supset \langle  \alpha^3\rangle  \supset\ldots \langle  \alpha^t\rangle \supset 0$ is a suitable $\alpha$-filtration.\\
By Lemma \ref{L3.17} e) it remains to consider the case $\langle \gamma \rangle \cap \langle \beta_1 \rangle =0$: Then $\langle  \beta_1 ,\gamma ,\alpha^2,\alpha\beta_1 \rangle =\langle  \beta_1 ,\alpha^2\rangle \oplus \langle \gamma ,\alpha\beta_1 \rangle $ by \ref{L3.20} b). Thus $\pd_{\Ld} \langle \gamma ,\alpha\beta_1 \rangle <\infty$. Now
$P_x \supset \langle \alpha,\beta_1 ,\gamma \rangle  \supset \langle \alpha^2\rangle \oplus \langle \gamma ,\alpha\beta_1 \rangle  \supset \langle \alpha^3\rangle  \supset\ldots \langle \alpha^t\rangle \supset 0$ is a suitable $\alpha$-filtration.
\end{itemize}
\end{proof}

\bibliographystyle{alpha}
\bibliography{references}
\end{document}